\documentclass[12pt]{amsart}
\usepackage{geometry}                % See geometry.pdf to learn the layout options. There are lots.
\usepackage{graphicx}
\usepackage{amssymb}
\usepackage{amsthm}
\usepackage{amsmath}
\usepackage{bbold}
\usepackage{epstopdf}
\def\E{{\bf E}}
\def\Z{{\mathcal{Z}}}
\def\C{{\mathbb{C}}}

\def\D{{\mathbb{D}}}
\def\Cir{{\mathcal{C}}}
\def\Ind{{\mathbb{1}}}

\title{On the Distribution of Critical Points of a Polynomial}
\author{Sneha Dey Subramanian}
\thanks{Ph.D. candidate, Department of Mathematics, University of Pennsylvania, 209 S. 33rd Street, Philadelphia, PA-19104, U.S.A. Email: ssneha@math.upenn.edu}
\begin{document}
\maketitle

\newtheorem*{thmnnum}{Theorem}
\newtheorem{prop2}{Proposition}[section]
\newtheorem*{propo}{Proposition}
\newtheorem{lem1}[prop2]{Lemma}
\newtheorem{thmnum}[prop2]{Theorem}
\newtheorem{lem2}{Lemma}[section]
\newtheorem{prop3}[lem2]{Proposition}

\begin{abstract}
This paper proves that if points $Z_1,Z_2,...$ are chosen independently and identically using some measure $\mu$ from the unit circle in the complex plane, with $p_n(z) = (z-Z_1)(z-Z_2)...(z-Z_n)$, then the empirical distribution of the critical points of $p_n$ converges weakly to $\mu$.
\end{abstract}
\section{Introduction}
Across many fields of mathematics, one of the fundamental questions about a function is the location of its zeros. Entire fields such as algebraic geometry and the emergent study of stable functions have locations of zeros as their focus.\\

The relation between the zeros of a function and the zeros of its derivative (the critical points) is interesting and not always obvious.  In the case where all zeros are real, Rolle's theorem tells us that the zeros of the derivative interlace the zeros of the function itself.  In the case of complex polynomials the analogous result is the Gauss-Lucas theorem which states that the zeros of the derivative of $f$ must lie in the convex hull of the zeros of $f$ and gives a representation of the zeros of $f'$ as convex combinations of the zeros of $f$. A corollary of this is that differentiating preserves stability.  Differentiation is also known never to increase the number of non-real zeros of a polynomial.\\

Two famous conjectures in this area are the conjectures of Sendov and Smale. The former, made by Blagovest Sendov during the 1950's, states that if the roots $z_1,z_2,...,z_n$ of a polynomial all lie inside the closed unit disc, then for each root of the polynomial, the closed unit disc centered at the root must contain at least one critical point. The latter, made by Steve Smale, states that if $f$ is a polynomial of degree $n$ with at least one root $0$ and $f'(0)\neq0$, then,
\begin{align*}
\min\left\{\frac{|f(\xi)|}{|\xi||f'(0)|}:f'(\xi)=0\right\}\leq K,
\end{align*}
where $K=1$ or $\frac{n-1}{n}$. Sendov's conjecture has been proven for the case when $z_1,z_2,...,z_n$ all lie on the unit circle, whereas Smale's conjecture has been proven for when $f$ has all its roots, save $0$, on the unit circle. The most general forms of these conjectures are still unsolved. More information on these conjectures and proofs of some of the special cases can be found in \cite{RahSch}.\\

Recent work in random marix theory has put forward numerous connections between the zeros and critical points of Riemann zeta function and those of the characteristic polynomial of a unitary matrix in the Circular Unitary Ensemble. While Keating and Snaith in \cite{keatsna} conjectured values for all even moments of Riemann zeta function on the critical line, Due\~nez et al. (\cite{randmat}) compared the horizontal distribution of critical points of the Riemann zeta function to the radial distribution of critical points of the characteristic polynomial of a random unitary matrix.\\

A probabilistic study on the roots of derivatives of polynomials was done by Pemantle and Rivin in \cite{PemRiv}.  Let $f$ be a polynomial with $n$ roots that are chosen independently and uniformly from a measure $\mu$ on the complex plane. They conjectured that the empirical distribution of the roots of $f'$ converges weakly to $\mu$ as $n \to \infty$.  They prove this in the special case when $\mu$ has finite 1-energy, namely when $\mu$ satisfies
\begin{align*}
\int\int\frac{1}{|z-w|}d\mu(z)d\mu(w) < \infty.
\end{align*}
This condition cannot hold, however, when $\mu$ is supported on any set of dimension 1 or less.  The aim of the present paper is to extend their result to the case of any measure supported on the unit circle.\\

The author would like to mention that while this paper was being refereed, a proof of the Pemantle-Rivin conjecture in the general case was found in \cite{Zakhar}, along very different lines from the approach taken here.\\

\section{Notations and Background}
Say, $Z_1,Z_2,...$ is a sequence of points chosen i.i.d. with respect to some distribution $\mu$ on the unit circle. Write, $Z_k = \exp(2\pi i \theta_k)$, so that $\{ \theta_k \}$ is a collection of IID random variables whose common law is supported on $[0,1]$, which we denote by $\nu$.\\

Let
\begin{align*}
p_n(z) = (z-Z_1)(z-Z_2)...(z-Z_n),
\end{align*}
and $y_1^{(n)}, y_2^{(n)},...,y_{n-1}^{(n)}$ be the roots of $p_n'(z)$.\\

For $k\geq1$, let $c_k = \E(Z^k)$, where $Z \sim \mu$. %Note that, $c_k=0, \forall k$, if and only if $\mu$ is uniform on the unit circle.
%We know, by the Strong Law of Large Numbers,
%\begin{align*}
%\frac{Z_1^k+Z_2^k+...Z_n^k}{n}\stackrel{a.s.}\longrightarrow c_k.
%\end{align*}
Denote by $\Z(f)$ the empirical distribution of the roots of a random polynomial $f$. That is, if $f$ has roots $X_1,X_2,...,X_m$, then $\Z(f) = \frac{1}{m}\sum_{j=1}^{m}\delta_{X_j}$.\\

We shall write $\D$ for the open unit disc, and $\Cir$ for the unit circle.\\

In their paper, \cite{PemRiv}, the authors conjectured that, for any distribution $\mu$ on the closed unit disc, $\Z(p_n')$ converges weakly to $\mu$. That paper also proves the following proposition.
\begin{prop2}\label{pemrivth}
Let $\mu$ be the uniform measure on ${\Cir}$.  Then ${\Z} (p_n')$ converges to ${\Cir}$ in probability, that is, $P({\Z} (S) \geq \epsilon) \to 0)$ for any $\epsilon > 0$ and any closed set $S\subset\D$, disjoint from $\Cir$.\qed
\end{prop2}

In this note, we shall generalize this to prove that
\begin{lem1}\label{pemrivgen}
For any distribution $\mu$ on $\Cir$, $\Z(p_n')$ converges to $\Cir$ in probability. In fact, if $\mu$ is not uniform on $\Cir$, the convergence is almost everywhere.
\end{lem1}
The above leads us to prove our main result, which is a special case of the aforementioned conjecture in \cite{PemRiv}:
\begin{thmnum}\label{main}
For any distribution $\mu$ on $\Cir$, $\Z(p_n')$ converges weakly to $\mu$ on $\Cir$.
\end{thmnum}
The proof, as shall be seen in forthcoming sections, can be divided in to two parts, the latter following a pattern similar to the proof of Weyl's equidistribution criterion (see, for example \cite{chandr}). The former requires the following theorem (proved both in \cite{komriv} and in \cite{ChNg}) regarding a companion matrix of the critical points.
\begin{propo}\label{companion}
If $z_1, z_2, ..., z_n \in \C$, and $y_1, y_2, ..., y_{n-1}$ are the critical points of the polynomial $p_n(z)=(z-z_1)(z-z_2)...(z-z_n)$, then, the matrix
\begin{align}
D\left(I-\frac{J}{n}\right)+\frac{z_n}{n}J
\end{align}
has $y_1, y_2, ..., y_{n-1}$ as its eigenvalues, where $D=diag(z_1,z_2,...,z_{n-1})$, $I$ is the identity matrix of order $n-1$ and $J$ is the $(n-1)\times(n-1)$ matrix of all entries 1.\qed\\\\
\end{propo}

\section{Proofs of Lemma \ref{pemrivgen} and Theorem \ref{main}}
We first begin by proving a small lemma.
\begin{lem2}
Let $\mu$ be a distribution on the unit circle $\Cir$ with $c_k = \E(Z^k)$, where $Z\sim\mu$. Then $c_k = 0$ for all $k \geq 1$ if and only if $\mu$ is uniform on $\Cir$.
\end{lem2}
\begin{proof}
Clearly if $\mu$ is uniform on $\Cir$ then $c_k=0$ for all $k\geq 1$. Now say $\mu$ is not uniform on the circle but we still have $c_k = 0$ for all $k\geq1$. Then the law $\nu$ is not uniform on $[0,1]$. Now, if $Z_1, Z_2, ...$ are points on $\Cir$, chosen i.i.d. using $\mu$, and if we write $Z_j = \exp(2\pi i \theta_j), j=1,2,...$, then $\theta_1, \theta_2,...$ are points in $[0,1]$ that are i.i.d. $\nu$.\\\\
By the Strong Law of Large Numbers, for all $k\geq1$,
\begin{align*}
\frac{Z_1^k+Z_2^k+...Z_n^k}{n}\stackrel{a.s.}\longrightarrow 0,
\end{align*}
and so by Weyl's criterion, for any $0\leq a < b \leq 1$,
\begin{align*}
\frac{\sum_{j=1}^n\Ind_{\{\theta_j\in[a,b]\}}}{n} \stackrel{a.s.}\longrightarrow b-a.
\end{align*}
But $\Ind_{\{\theta_j\in[a,b]\}}, j=1,2,...$ are i.i.d. random variables taking values 0 or 1 with expectation $\nu([a,b])$. Therefore,
\begin{align*}
\frac{\sum_{j=1}^n\Ind_{\{\theta_j\in[a,b]\}}}{n} \stackrel{a.s.}\longrightarrow \nu([a,b]).
\end{align*}
Since $\nu$ is not uniform on $[0,1]$, we have arrived at a contradiction. So, there must exist at least one non-zero $c_k$.
\end{proof}

We proceed to use this fact for the proof of Lemma \ref{pemrivgen}.\\
\begin{proof}[Proof of Lemma \ref{pemrivgen}]
Assume $\mu$ is not the uniform distribution on the circle (as the uniform case has been taken care of in \cite{PemRiv}). Then, as mentioned above, there is at least one non-zero $c_k$. Thus the power series function $f(z) = \sum_{k=0}^\infty \bar{c}_{k+1} z^k$ exists at every point $z\in\D$, is analytic there (since $|c_k|<1, \forall k$), and so has only finitely many zeros inside any $r$-ball, where $r<1$.\\

Define
\begin{align*}
V_n(z) = \frac{p_n'(z)}{np_n(z)} = \frac{1}{n}\sum_{j=1}^n\frac{1}{z-Z_j}.
\end{align*}
$V_n$ has $n-1$ zeros, which are exactly the zeros of $p_n'(z)$, and $n$ poles, which are exactly the zeros of $p_n(z)$. Thus $V_n(z)$ is analytic inside $\D$. We shall show that as $n\to\infty$, $V_n$ converges inside the disc to $-f$, uniformly over compact sets. To see this, note that for $z\in\D$,
\begin{align*}
V_n(z) &= \frac{1}{n}\sum_{j=1}^n\frac{-1/Z_j}{1-z/Z_j} %= -\frac{1}{n}\sum_{j=1}^n\frac{1}{Z_j}\left(1+\frac{z}{Z_j}+\frac{z^2}{Z_j^2}+...\right)\\
= -\frac{1}{n}\sum_{j=1}^n\sum_{k=0}^\infty\bar{Z}_j^{k+1}z^k = -\sum_{k=0}^\infty \bar{a}_n^{k+1}z^k,
\end{align*}
where, we write $a_n^{k+1}$ for the $k$th power sum average $\frac{Z_1^k+Z_2^k+...+Z_n^k}{n}$. By Strong Law of Large Numbers, $a_n^k \stackrel{a.s.}\longrightarrow c_k$ for all $k\geq 1$.\\

Let $0<r<1$. Given any $\delta>0, \exists K\geq1$ such that
\begin{align*}
\sum_{k=K}^\infty r^k = \frac{r^k}{1-r} < \frac{\delta}{4}.
\end{align*}
Corresponding to the chosen $K$, there exists an $N\geq1$ such that,
\begin{align*}
|a_n^k - c_k| < \frac{\delta(1-r)}{2},
\end{align*}
$\forall n\geq N$ and $\forall k=1,2,...,K-1$. Therefore, $\forall n\geq N$ and all $z\in B_r(0)$,
\begin{align*}
|V_n(z) + f(z)| &\leq \sum_{k=0}^{K-1}|a_n^k - c_k|r^k + \sum_{k=K+1}^{\infty}|a_n^k - c_k|r^k\\
&\leq \frac{\delta(1-r)}{2}\cdot(1+r+r^2+...+r^{K-1}) + 2\cdot\frac{\delta}{4} < \delta,
\end{align*}
which proves uniform convergence of $V_n$ to $-f$ over compact sets.\\

%Hurwitz's theorem (see, for example, \cite{Conw}) says that, if $g_n$ is a sequence of analytic functions that converges uniformly on compact subsets of an open set $G$ in the complex plane to an analytic function $g$, then, for any closed disc $D$ contained in $G$, if $g$ has no zeros on the boundary of $D$, there shall exist a natural number $N$ for which $g_n$ and $g$ have the same number of zeros in the interior of $D$, for all $n\geq N$.\\

Using Hurwitz's theorem (see \cite{Conw}), given any $0<r<1$, there exists an $M\geq1$ for which $V_n$ and $f$ have the same number of zeros inside $B_r(0)$ for all $n\geq M.$ That is, $p_n'$ and $f$ shall have the same number of zeros inside $B_r(0)$ for all $n\geq M.$ But, as discussed above, $f$ has only finitely many zeros inside $B_r(0)$. Thus $\Z(p_n')$ converges to the unit circle almost surely.
\end{proof}

Our main result, Theorem \ref{main}, will be a consequence of the following proposition.\\
\begin{prop3}\label{myprop}
Given any sequence of points $z_1,z_2,...$ with $|z_n|\leq M$ for all $n$, and $\frac{z_1^k+z_2^k+...z_n^k}{n}\to c_k$ as $n\to\infty$, $\forall k\geq 1$, the critical points $y_1^{(n)}, y_2^{(n)},...,y_{n-1}^{(n)}$ of $p_n(z)=(z-z_1)(z-z_2)...(z-z_n)$ also satisfy
\begin{align*}
\frac{(y_1^{(n)})^k+ (y_2^{(n)})^k+...+(y_{n-1}^{(n)})^k}{n-1}\longrightarrow c_k\text{ as }n\to\infty,
\end{align*}
$\forall k\geq1$.
\end{prop3}
\begin{proof}
Note that, it is easy to see that this theorem holds true for $k=1$, because the average of the critical points is exactly equal to the average of the roots (by comparing the coefficients of $z^{n-1}$ in $p_n(z)$ with $z^{n-2}$ of $p_n'(z)$). To prove the result for general $k$, we use a result of \cite{komriv} (also appeared in \cite{ChNg}), mentioned as a proposition in Section 2, to see that for $k\geq2$, $(y^{(n)}_1)^k,(y^{(n)}_2)^k,...,(y^{(n)}_{n-1})^k$ are the eigenvalues of $[D\left(I-\frac{1}{n}J\right)+\frac{z_n}{n}J]^k$, and so,
\begin{align*}
(y^{(n)}_1)^k+(y^{(n)}_2)^k+...+(y^{(n)}_{n-1})^k = Tr\left[D\left(I-\frac{1}{n}J\right)+\frac{z_n}{n}J\right]^k.
\end{align*}
Note that the expansion of $[D\left(I-\frac{1}{n}J\right)+\frac{z_n}{n}J]^k$ is the sum of all terms such as
\begin{align}\label{ugly}
D^{l_1}\left(-\frac{DJ}{n}\right)^{l_2}\left(\frac{z_n}{n}J\right)^{l_3}D^{l_4}\left(-\frac{DJ}{n}\right)^{l_5}\left(\frac{z_n}{n}J\right)^{l_6}...D^{l_{3k-2}}\left(-\frac{DJ}{n}\right)^{l_{3k-1}}\left(\frac{z_n}{n}J\right)^{l_{3k}}
\end{align}
where the exponents $l_1,l_2,...,l_{3k}$ are non-zero integers, with $l_{3j-2}+l_{3j-1}+l_{3j}=1$ for all $j=1,2,..,k$. Clearly the number of such terms is $3^k$, which does not depend on $n$, and so, if we find that the trace of the matrix in the expression (\ref{ugly}) converges as $n\to\infty$ to $a_{l_1,l_2,...l_{3k}}$, then the trace of $[D\left(I-\frac{1}{n}J\right)+\frac{z_n}{n}J]^k$ converges to $\sum a_{l_1,l_2,...,l_{3k}}$.\\\\
Henceforth, we fix $l_1,l_2,...l_{3k}$. Now, note that $J^m = (n-1)^{m-1}J^{m-1}$ for any $m\geq 1$, and
\begin{align*}
(D^pJ)(D^qJ) &=
%\begin{pmatrix}
%z_1^p & z_1^p & \cdots & z_1^p\\
%z_2^p & z_2^p & \cdots & z_2^p\\
%\vdots & \vdots & \ddots & \vdots\\
%z_{n-1}^p & z_{n-1}^p & \cdots & z_{n-1}^p
%\end{pmatrix}
%\cdot
%\begin{pmatrix}
%z_1^q & z_1^q & \cdots & z_1^q\\
%z_2^q & z_2^q & \cdots & z_2^q\\
%\vdots & \vdots & \ddots & \vdots\\
%z_{n-1}^q & z_{n-1}^q & \cdots & z_{n-1}^q
%\end{pmatrix}\\
\left(\sum_{i=1}^{n-1}z_i^q\right) (D^pJ),
\end{align*}
for any $p,q\geq 0$.\\

The above tells us that there exists $p,q,s_0,s_1,s_2,...,s_{k-1}\geq0$ such that, term (\ref{ugly}) is of the form
\begin{align}\label{genterm}
(-1)^p\cdot z_n^q\cdot\left(\frac{n-1}{n}\right)^{s_0}\cdot\left(\frac{\sum_{i=1}^{n-1}z_i}{n}\right)^{s_1}\cdot\left(\frac{\sum_{i=1}^{n-1}z_i^2}{n}\right)^{s_2}\cdot...\cdot \left(\frac{\sum_{i=1}^{n-1}z_i^{k-1}}{n}\right)^{s_{k-1}}\cdot M,
\end{align}
where the numbers $p,q,s_0,s_1,...,s_{k-1}$ are determined solely by the $l_i$'s (and so, are independent of $n$).\\
 
Also, $M$ can only be one of the following terms: $D^k$ or $\frac{D^mJ}{n}$ or $\frac{D^{m_1}J}{n}D^{m_2}$ for some $m,m_1,m_2\geq 0$, which are fixed, $\leq k$, and dependent only on the $l_i$'s. Furthermore, the scalar coefficient in (\ref{genterm}) is always $O(1)$.\\

Observe that, if $M = D^k$, then the scalar coefficient in (\ref{genterm}) is equal to 1 and $\frac{Tr(M)}{n}\to c_k$. On the other hand, if $M=\frac{D^mJ}{n}$, then
\begin{align*}
Tr(M) = \frac{z_1^m+z_2^m+...+z_{n-1}^m}{n} = o(n),
\end{align*}
and if $M=\frac{D^{m_1}J}{n}D^{m_2}$,
\begin{align*}
Tr(M) &= Tr\left(D^{m_1+m_2}\frac{J}{n}\right)\\
&= \frac{z_1^{m_1+m_2}+ z_2^{m_1+m_2}+...+z_{n-1}^{m_1+m_2}}{n} = o(n).
\end{align*}

Thus,
\begin{align*}
\frac{Tr\left[D\left(I-\frac{1}{n}J\right)+\frac{z_n}{n}J\right]^k}{n} \longrightarrow c_k \text{ as }n\to\infty.
\end{align*}

\end{proof}

%We require now a simple lemma about the convergence of weighted averages with some suitable conditions on the weights.
%\begin{lem2}\label{simplem}
%Let $\{x_{nk}\}_{n.k=1}^\infty$ be a double sequence of real numbers or real-valued random variables, for which, $x_{nk}\in(0,1)$ for all $n$ and for all $k$. Also, let $\{b_{nk}\}$ be a double sequence of complex-valued random variables, with $|b_{nk}|\leq M$ for all $n$ and for all $k$. Then, if 
%\begin{align*}
%\frac{x_{1n}+x_{2n}+...+x_{nn}}{n}\to0,
%\end{align*}
%in probability, almost surely, or pointwise,
%\begin{align*}
%\frac{x_{1n}b_{1n}+x_{2n}b_{2n}+...+x_{nn}b_{nn}}{n}\to0,
%\end{align*}
%respectively in probability, almost surely, or pointwise.
%\end{lem2}
%\begin{proof}
%We have,
%\begin{align*}
%\left|\frac{x_{1n}b_{1n}+x_{2n}b_{2n}+...+x_{nn}b_{nn}}{n}\right|&\leq \frac{x_{1n}|b_{1n}|+x_{2n}|b_{2n}|+...+x_{nn}|b_{nn}|}{n}\\
%&\leq M\left(\frac{x_{1n}+x_{2n}+...+x_{nn}}{n}\right),
%\end{align*}
%which gives us the result.
%\end{proof}

We now have all the tools required to prove our main result, namely Theorem \ref{main}.\\
\begin{proof}[Proof of Theorem \ref{main}]
Say we write,
\begin{align*}
y_j^{(n)} = r_j^{(n)}\exp(2\pi i \phi_j^{(n)}), j=1,2,...,n-1.
\end{align*}
The proof will consist of three major segments. Our first task is to prove that
\begin{align*}
\frac{1}{n-1}\sum_{j=1}^{n-1}(r_j^{(n)})^k\stackrel{P}\longrightarrow 1.
\end{align*}
In fact, unless $\mu$ is uniform on the circle, we will show that
\begin{align*}
\frac{1}{n-1}\sum_{j=1}^{n-1}(r_j^{(n)})^k\stackrel{a.s.}\longrightarrow 1.
\end{align*}
Next, we shall use the above information to show that
\begin{align*}
\frac{\exp(2k\pi i\phi_1^{(n)})+\exp(2k\pi i\phi_2^{(n)})+...+\exp(2k\pi i\phi_{n-1}^{(n)})}{n-1}\stackrel{P}\longrightarrow c_k.
\end{align*}
(Again, the convergence is almost sure, unless $\mu$ is uniform on $\Cir$.)\\

Finally, using arguments analogous to those in the proof of Weyl's equidistribution criterion, we shall arrive at our final result.\\

Assume, initially, that $\mu$ is not the uniform law on $\Cir$. For the first task as noted above, observe that, by Lemma \ref{pemrivgen}, given any $\epsilon > 0$,
\begin{align*}
\frac{1}{n-1}\sum_{j=1}^{n-1}\Ind_{\{r_j^{(n)}\in[1-\epsilon,1]\}}\stackrel{a.s.} \longrightarrow 1.
\end{align*}

Now, for any fixed positive integer $k$, $(1-\epsilon)^k\Ind_{\{r_j^{(n)}\in[1-\epsilon,1]\}}\leq (r_j^{(n)})^k \leq 1,$ and so
\begin{align}\label{inequ}
(1-\epsilon)^k \cdot \frac{1}{n-1}\sum_{j=1}^{n-1}\Ind_{\{r_j^{(n)}\in[1-\epsilon,1]\}} \leq \frac{1}{n-1}\sum_{j=1}^{n-1}(r_j^{(n)})^k \leq 1.
\end{align}

Clearly then, a simple squeeze theorem argument gives us
\begin{align}\label{raver}
\frac{1}{n-1}\sum_{j=1}^{n-1}(r_j^{(n)})^k\stackrel{a.s.}\longrightarrow 1.
\end{align}

%If $\mu$ were uniform on the circle, then, for any $\delta>0$,
%\begin{align*}
%\P\left(1-\frac{1}{n-1}\sum_{j=1}^{n-1}(r_j^{(n)})^k \geq\delta\right) &\leq \P\left(1-(1-\epsilon)^k \cdot \frac{1}{n-1}\sum_{j=1}^{n-1}\Ind_{\{r_j^{(n)}\in[1-\epsilon,1]\}}\geq\delta\right)\\
%&\leq \P\left(1- \frac{1}{n-1}\sum_{j=1}^{n-1}\Ind_{\{r_j^{(n)}\in[1-\epsilon,1]\}}\geq\frac{\delta-1+(1-\epsilon)^k}{(1-\epsilon)^k}\right).
%\end{align*}
%We may choose $\epsilon$ in such a way that $1-(1-\epsilon)^k < \delta/2,$ so that, the left hand side of the above inequality goes to 0. Thus, for any positive integer $k$,
%\begin{align}
%\frac{1}{n-1}\sum_{j=1}^{n-1}(r_j^{(n)})^k\stackrel{P}\longrightarrow 1.
%\end{align}

Now, from Proposition \ref{myprop}, for any positive integer $k$,
\begin{align*}
&\frac{(y_1^{(n)})^k+(y_2^{(n)})^k+...+(y_{n-1}^{(n)})^k}{n-1} \stackrel{a.s.}\longrightarrow c_k,\\
\implies&\frac{(r_1^{(n)})^k\exp(2k\pi i\phi_1^{(n)})+(r_2^{(n)})^k\exp(2k\pi i\phi_2^{(n)})+...+(r_{n-1}^{(n)})^k\exp(2k\pi i\phi_{n-1}^{(n)})}{n-1}\stackrel{a.s.}\longrightarrow c_k.
\end{align*}

Note that (\ref{raver}) gives us that
\begin{align*}
\left|\frac{1}{n-1}\sum_{j=1}^{n-1}(1-(r_j^{(n)})^k)\exp(2k\pi i\phi_j^{(n)})\right|\leq \frac{1}{n-1}\sum_{j=1}^{n-1}(1-(r_j^{(n)})^k)\stackrel{a.s.}\longrightarrow 0,
\end{align*}
and so,
\begin{align}\label{weyluse}
\frac{\exp(2k\pi i\phi_1^{(n)})+\exp(2k\pi i\phi_2^{(n)})+...+\exp(2k\pi i\phi_{n-1}^{(n)})}{n-1}\stackrel{a.s.}\longrightarrow c_k.
\end{align}

Now, for the final stage of our proof,
\begin{align*}
c_k &= \E(Z^k), \text{ where, } Z\sim\mu.\\
\implies c_k &= \E(\exp(2k\pi i \Theta)) = \E(\cos(2k\pi\Theta)) + i\E(\sin(2k\pi\Theta)), \text{ where, } \Theta\sim\nu.
\end{align*}
So, (\ref{weyluse}) gives,
\begin{align*}
\frac{\cos(2k\pi\phi_1^{(n)})+\cos(2k\pi\phi_2^{(n)})+...+\cos(2k\pi\phi_{n-1}^{(n)})}{n}\stackrel{a.s.}\longrightarrow \E(\cos(2k\pi\Theta)),\\
\frac{\sin(2k\pi\phi_1^{(n)})+\sin(2k\pi\phi_2^{(n)})+...+\sin(2k\pi\phi_{n-1}^{(n)})}{n}\stackrel{a.s.}\longrightarrow \E(\sin(2k\pi\Theta)).
\end{align*}
Then, for any trigonometric polynomial $q(x)$,
%\begin{align*}
%q(x) = a_0 + \sum_{j=1}^r(a_j \cos(2\pi jx) + b_j \sin(2\pi jx)), a_j,b_j\in\R, j= 1,2, ...,r,
%\end{align*}
%we have,
\begin{align}\label{trig}
\frac{\sum_{j=1}^{n-1}q(\phi_j^{(n)})}{n} \stackrel{a.s.}\longrightarrow \E(q(\Theta)).
\end{align}

Let $f$ be a continuous real-valued function on [0, 1] and fix $\epsilon > 0$. By Stone-Weierstrass theorem (\cite{Stone}), there exists a trigonometric polynomial $q$ such that $|f - q| < \epsilon$. So,
\begin{align*}
\left|\frac{\sum_{j=1}^{n-1}f(\phi_j^{(n)})}{n} - \E(f(\Theta))\right| &\leq \left|\frac{\sum_{j=1}^{n-1}f(\phi_j^{(n)})}{n} - \frac{\sum_{j=1}^{n-1}q(\phi_j^{(n)})}{n}\right|\\
& + \left| \frac{\sum_{j=1}^{n-1}q(\phi_j^{(n)})}{n} - \E(q(\Theta))\right| + \E|q(\Theta)-f(\Theta)|.
\end{align*}
The first and third terms on the right hand side are each $<\epsilon$ while the second term goes to 0 almost surely, by (\ref{trig}). Hence for any $f$ continuous on [0, 1],
\begin{align}\label{finalnuni}
\frac{\sum_{j=1}^{n-1}f(\phi_j^{(n)})}{n}\stackrel{a.s.}\longrightarrow \E(f(\Theta)),
\end{align}
and this holds for complex-valued continuous functions as well (which is easily seen by comparing the real and imaginary parts). Thus, the joint empirical distribution of $\phi_j^{(n)}, j=1,2,...,n-1,$ converges weakly to $\nu$, which means that the joint empirical distribution of $\exp(2\pi i\phi_j^{(n)}), j=1,2,...,n-1,$ converges weakly to $\mu$. This, along with Lemma \ref{pemrivgen}, gives us the desired result for $\mu$ not uniform on $\Cir$.\\

Now suppose $\mu$ is the uniform law on the unit circle. Then,
\begin{align*}
\frac{1}{n-1}\sum_{j=1}^{n-1}\Ind_{\{r_j^{(n)}\in[1-\epsilon,1]\}}\stackrel{P} \longrightarrow 1,
\end{align*}
and as before, using (\ref{inequ}) we get,
% by (\ref{inequ}),
%for any $\delta>0$,
%\begin{align*}
%\P\left(1-\frac{1}{n-1}\sum_{j=1}^{n-1}(r_j^{(n)})^k \geq\delta\right) &\leq \P\left(1-(1-\epsilon)^k \cdot \frac{1}{n-1}\sum_{j=1}^{n-1}\Ind_{\{r_j^{(n)}\in[1-\epsilon,1]\}}\geq\delta\right)\\
%&\leq \P\left(1- \frac{1}{n-1}\sum_{j=1}^{n-1}\Ind_{\{r_j^{(n)}\in[1-\epsilon,1]\}}\geq\frac{\delta-1+(1-\epsilon)^k}{(1-\epsilon)^k}\right).
%\end{align*}

%We may choose $\epsilon$ in such a way that $1-(1-\epsilon)^k < \delta/2,$ and so, the right hand side of the above inequality goes to 0. Thus, for any positive integer $k$,
\begin{align*}
\frac{1}{n-1}\sum_{j=1}^{n-1}(r_j^{(n)})^k\stackrel{P}\longrightarrow 1,
\end{align*}
for any positive integer $k$.\\

Note that the above is a slightly weaker version of (\ref{raver}), since the convergence is now in probability, and not almost sure.\\

For the rest of the proof, we can follow the same arguments as in the non-uniform case, except that the almost sure convergence in each of the statements will be replaced by convergence in probability. Thus we shall arrive at
\begin{align*}
\frac{\sum_{j=1}^{n-1}f(\phi_j^{(n)})}{n}\stackrel{P}\longrightarrow \E(f(\Theta)),
\end{align*}
for any continuous function $f:[0,1]\to\C$. Then, as before, the joint empirical distribution of $\phi_j^{(n)}, j=1,2,...,n-1,$ converges weakly to $\nu$ (which is the uniform law on $[0,1]$), and so, the joint empirical distribution of $\exp(2\pi i\phi_j^{(n)}), j=1,2,...,n-1,$ converges weakly to uniform on $\Cir$. Lemma \ref{pemrivgen} then gives us the desired result.
%Taking $f_1 = q - \frac{\epsilon}{2}$ and $f_2 = q + \frac{\epsilon}{2}$ , we have $f_1 \leq f \leq f_2$ and $\int(f_2 - f_1)d\nu = \epsilon$.
\end{proof}

\vspace{8 mm}
\section*{Acknowledgements}
The author is grateful to Robin Pemantle, Philip Gressman and Andreea Nicoara for stimulating discussions and helpful suggestions.
\vspace{8 mm}

%\subsection{}


\begin{thebibliography}{10}
\bibitem[Ch68]{chandr} K. Chandrasekharan. {\it Introduction to Analytic Number Theory}. Springer-Verlag, 1968.\\
\bibitem[CN06]{ChNg} W. S. Cheung and T. W. Ng. A companion matrix approach to the study of zeros and critical points of a polynomial. {\it J. Math. Anal. Appl.}, 319(2): 690-707, 2006.\\
\bibitem[Co78]{Conw} John B. Conway. {\it Functions of One Complex Variable}. Springer-Verlag, 1978.\\
\bibitem[dB46]{bruijn1} N. G. de Bruijn. On the zeros of a polynomial and of its derivative. {\it Indag. Math.}, 8: 635-643, 1946.\\
\bibitem[dBS47]{bruijn2} N. G. de Bruijn and T. A. Springer. On the zeros of a polynomial and of its derivative. II. {\it Indag. Math.}, 9: 264-270, 1947.\\
\bibitem[DFFHMP10]{randmat} E. Due\~nez, D. W. Farmer, S. Froehlich, C. Hughes, F. Mezzadri and T. Phan. Roots of the derivative of the Riemann zeta function and of characteristic polynomials. {\it Nonlinearity}, 23: 2599-2621, 2010.\\
\bibitem[Za12]{Zakhar} Z. Kabluchko. Critical points of random rolynomials with independent identically distributed roots. {\it Preprint: http://arxiv.org/pdf/1206.6692v1.pdf}\\
\bibitem[KS00]{keatsna} J. P. Keating and N. C. Snaith. Random Matrix Theory and $\zeta(1/2 + it)$. {\it Commun. Math. Phys.}, 214: 57 Ð 89, 2000.\\
\bibitem[KR01]{komriv} N. Komarova and I. Rivin. Harmonic Mean, Random Polynomials and Stochastic Matrices. {\it Adv. in Appl. Math.} 31(2): 501-526, 2003.\\
\bibitem[Pe12]{Pemstab} R. Pemantle. Hyperbolicity and stable polynomials in combinatorics and probability. {\it Preprint: http://www.math.upenn.edu/~pemantle/papers/Preprints/hyperbolic.pdf}\\
\bibitem[PR12]{PemRiv} R. Pemantle and I. Rivin. The distribution of the zeroes of the derivative of a random polynomial. {\it{Preprint: http://www.math.upenn.edu/~pemantle/papers/Preprints/zeros.pdf}}\\
\bibitem[RS02]{RahSch} Q. I. Rahman and G. Schmeisser. {\it Analytic Theory of Polynomials}. Oxford University Press, Oxford, 2002.\\
\bibitem[St48]{Stone} M. H. Stone. The Generalized Weierstrass Approximation Theorem. {\it Mathematics Magazine}, 21(4): 167-184, 1948.\\
\end{thebibliography}
\end{document}